\documentclass[11pt]{article}
\usepackage{amsmath, amsthm, amssymb, pdfsync, psfrag, verbatim}
\usepackage {mathrsfs, graphicx, newcent}

\newcommand{\be}{\begin{equation}}
\newcommand{\ee}{\end{equation}}

\newcommand{\bbE}{\mathbb E}
\newcommand{\bbF}{\mathbb F}

\newcommand{\bbH}{\mathbb H}

\newcommand{\bbN}{\mathbb N}

\newcommand{\bbP}{\mathbb P}

\newcommand{\bbR}{\mathbb R}
\newcommand{\bbS}{\mathbb S}

\newcommand{\scB}{\mathcal B}

\newcommand{\scF}{\mathcal F}

\newcommand{\scT}{\mathcal T}

\newcommand{\veps}{\varepsilon}

\newcommand{\norm}[1]{\ensuremath{\left\| #1 \right\|}}
\newcommand{\abs}[1]{\ensuremath{\left| #1 \right|}}

\DeclareMathOperator*{\esssup}{ess\,sup}

\newcommand{\Ito}{It\ensuremath{\hat{\textrm{o}}}}

\newcommand{\crl}[1]{\ensuremath{ \left\{ #1 \right\} }}
\newcommand{\edg}[1]{\ensuremath{ \left[ #1 \right] }}
\newcommand{\brak}[1]{\ensuremath{\left( #1 \right)}}

\newtheorem{theorem}{Theorem}[section]
\newtheorem{definition}[theorem]{Definition}
\newtheorem{proposition}[theorem]{Proposition}

\newtheorem{lemma}[theorem]{Lemma}
\newtheorem{remark}[theorem]{Remark}
\newtheorem{example}[theorem]{Example}
\newtheorem{examples}[theorem]{Examples}
\newtheorem{foo}[theorem]{Remarks}

\newenvironment{Remark}{\begin{remark}\rm}{\end{remark}}

\title{Multidimensional quadratic and subquadratic BSDEs with special structure}
\author{
Patrick Cheridito\\
Princeton University\\
Princeton, NJ 08544, USA
\and
Kihun Nam\\
Princeton University\\
Princeton, NJ 08544, USA
}
\date{January 2015}

\begin{document}
\maketitle

\begin{abstract}
We study multidimensional BSDEs of the form
$$
Y_t = \xi + \int_t^T f(s,Y_s,Z_s)ds - \int_t^T Z_s dW_s
$$ 
with bounded terminal conditions $\xi$ and drivers $f$ that grow at most quadratically in $Z_s$.
We consider three different cases. In the first one the BSDE is Markovian, and 
a solution can be obtained from a solution to a related FBSDE. In the second case, 
the BSDE becomes a one-dimensional quadratic BSDE when projected 
to a one-dimensional subspace, and a solution can be derived from 
a solution of the one-dimensional equation. 
In the third case, the growth of the driver $f$ in $Z_s$ is strictly subquadratic, and 
the existence and uniqueness of a solution can be shown by first solving the BSDE on 
a short time interval and then extending it recursively.\\[2mm]
\end{abstract}

\setcounter{equation}{0}
\section{Introduction}

We study multidimensional BSDEs of the form 
\begin{align}\label{bsde}
Y_t=\xi + \int_t^T f(s,Y_s,Z_s) ds - \int_t^T Z_s dW_s, 
\end{align}
where $Y_t$ takes values in $\mathbb{R}^d$, $Z_t$ is $\mathbb{R}^{d \times n}$-valued
and $W$ is an $n$-dimensional Brownian motion.
If the terminal condition $\xi$ is square-integrable and the driver 
$f(t,y,z)$ Lipschitz continuous in $(y,z)$, the existence of a unique solution 
can be shown with a Picard--Lindel\"of iteration argument, see
e.g. Pardoux and Peng (1990) or El Karoui et al. (1997).
Kobylanski (2000) proved that one-dimensional ($d=1$) BSDEs with drivers of 
quadratic growth in $z$ have solutions if $\xi$ is bounded.  
Moreover, if $\xi$ has bounded Malliavin derivative, the growth of
$f(s,y,z)$ in $z$ can be arbitrary; see Cheridito and Nam (2014).
For multidimensional BSDEs the situation is more complicated because one cannot use 
comparison results; see e.g. Hu and Peng (2006). In fact, multidimensional 
BSDEs with drivers of quadratic growth in $z$ do not always admit solutions even 
if the terminal condition $\xi$ is bounded; see Frei and dos Reis (2011) for an example. 
An early result for superlinear multidimensional BSDEs was given by Bahlali et
al. (2001), which assumed that the growth of $f(s,y,z)$ in $z$ is of the order 
$|z|\sqrt{\log |z|}$. It was generalized by Bahlali et al. (2010) to the case
where $f(s,y,z)$ has strictly subquadratic growth in $z$ and satisfies a
monotonicity condition. Tevzadze (2008) gave an existence and uniqueness result 
for multidimensional BSDEs with general drivers of quadratic $z$-growth in the case 
where the terminal condition has small $L^{\infty}$-norm.

In this paper we put restrictions on the driver. Three different cases are considered.
In all three we assume $\xi$ to be bounded and use BMO-martingale theory 
together with Girsanov's theorem to construct an equivalent probability measure 
that can be used to prove the existence of a solution.

In Section \ref{sec:FBSDE} we assume the BSDE to be Markovian and related to an FBSDE of the form 
\be \label{fbsde}
\begin{aligned}
dP_t & = G(t,P_t,Q_t,R_t)dt+dW_t, \quad P_0=0\\
dQ_t & =-F(t,P_t,Q_t,R_t)dt+ R_t dW_t, \quad Q_T= h(P_T)
\end{aligned}
\ee
for a bounded function $h$. If the FBSDE has a solution, we change the 
probability measure to obtain a solution to a different FBSDE, from which a solution to the BSDE \eqref{bsde} can 
be derived. A similar approach was taken by Liang et al. (2010) but without proving that
the solution is adapted to the filtration generated by the Brownian motion driving the BSDE;
that is, they showed the existence of a weak solution. Here, we prove the existence of a strong solution.
There exist different results in the literature guaranteeing the existence of a solution to 
\eqref{fbsde}. We use one of Delarue (2002) to derive that the BSDE \eqref{bsde} has a unique 
bounded solution. Mania and Schweizer (2005) and Ankirchner et al. (2009) studied the transformation of 
one-dimensional quadratic BSDEs under a change of measure, but not with the aim
of proving the existence of a classical solution. In Section \ref{sec:proj} conditions 
are given under which equation \eqref{bsde} can be turned into a one-dimensional quadratic BSDE
by projecting it on a one-dimensional subspace of $\mathbb{R}^d$.
Results of Kobylanski (2000) guarantee that the one-dimensional 
equation has a solution. From there a solution to the multidimensional equation can
be obtained by changing the probability measure and solving a linear equation.
In Section \ref{sec:sub} the growth of $f(s,y,z)$ in $z$ is assumed to be strictly subquadratic.
This makes it possible to prove the existence of a unique solution on a short time interval with a 
contraction argument. Under an additional structural assumption, the solution can 
be estimated by taking conditional expectation with respect to an
equivalent probability measure. Then the short-time solution can be extended 
to a global solution.\\[2mm]
{\bf Notation:}\\
In the whole paper $T \in \mathbb{R}_+$ is a finite time horizon and 
$(W_t)_{0 \le t\le T}$ an $n$-dimensional Brownian motion on a 
probability space $(\Omega,\scF,\bbP)$. By $\bbF$ we denote the augmented 
filtration $({\cal F}_t)$ generated by $W$.  The terminal condition $\xi$ is a
bounded $d$-dimensional $\scF_T$-measurable random vector 
and the driver $f : [0,T] \times \Omega \times \mathbb{R}^d \times \mathbb{R}^{d \times n} \to \mathbb{R}^d$
a ${\cal P} \otimes {\cal B}(\mathbb{R}^d) \otimes {\cal B}(\mathbb{R}^{d \times n})$-measurable 
mapping, where ${\cal P}$ denotes the predictable sigma-algebra and 
${\cal B}(\mathbb{R}^d)$ and ${\cal B}(\mathbb{R}^{d \times n})$ the Borel sigma-algebras
on $\bbR^d$ and $\bbR^{d \times n}$, respectively. 
As usual, we understand equalities and inequalities between random variables in the
$\bbP$-almost sure sense. $Y, \xi$ and $f$ are understood as $d \times 1$-matrices, $W$ as an $n \times 1$-matrix
and $Z$ as a $d \times n$-matrix. 
By $Z^T$ we denote the transpose of $Z$ and by $|\cdot|$ the Euclidean norm. 
That is, for a vector $Y$, $|Y| := \sqrt{\sum_i Y^2_i}$, and for a 
matrix $Z$, $|Z|=\sqrt{{\rm tr}(ZZ^T)}$. $\bbR_+$ denotes the set of nonnegative real numbers
and $\|.\|_p$ the $L^p$-norm. We need the following Banach spaces of stochastic processes:
\begin{itemize}
\item $\bbS^p(\bbR^d)$: all $d$-dimensional continuous adapted processes satisfying
$$
\|Y\|_{\bbS^p}:= \norm{\sup_{0 \le t \le T}|Y_t|}_p <\infty
$$
\item $\bbH^p(\bbR^{d \times n})$ for $p < \infty$: all $\bbR^{d \times n}$-valued predictable processes satisfying 
$$
\norm{Z}_{\bbH^p}:= \norm{\brak{\int_0^T |Z_s|^2 ds}^{1/2}}_p <\infty
$$
\item $\bbH^{\infty}(\bbR^{d \times n})$: all $\bbR^{d \times n}$-valued predictable processes satisfying 
$$
\norm{Z}_{\bbH^{\infty}}:= \esssup_{(t, \omega)} |Z_t(\omega)| <\infty
$$
\item $\bbH^{\rm BMO}(\bbR^{d \times n})$: all $Z \in \bbH^2(\bbR^{d \times n})$ satisfying
$$
\|Z\|_{\rm BMO} := \sup_{\tau\in\scT}
\left\|\bbE_\tau \int_{\tau}^T |Z_s|^2 ds \right\|^{1/2}_{\infty} < \infty,
$$
where $\scT$ denotes the set of all $[0,T]$-valued stopping times $\tau$
and $\bbE_{\tau}$ the conditional expectation with respect to ${\cal F}_{\tau}$.
\end{itemize}

By $\bbS^p_{[a,b]}(\bbR^d)$, $\bbH^p_{[a,b]}(\bbR^{d\times n})$ and
$\bbH^{\rm BMO}_{[a,b]}(\bbR^{d \times n})$ we denote the same spaces if the processes 
have time indexes in $[a,b]$.

For $H \in \bbH^{\rm BMO}(\bbR^{n \times 1})$, $\int_0^t H^T_s dW_s$ is a BMO-martingale and
$$
{\cal E}^H_t := \exp \brak{\int_0^t H^T_s dW_s - \frac{1}{2} \int_0^T |H_s|^2 ds}
$$
a martingale; see Kazamaki (1994). So one obtains from Girsanov's theorem that 
${\cal E}^H_T \cdot \bbP$ defines an equivalent probability measure, under which 
$W_t - \int_0^t H_s ds$ is a Brownian motion. Moreover, every $Z \in \bbH^{\rm BMO}(\bbR^{d \times n})$ 
with respect to $\bbP$ is also in $\bbH^{\rm BMO}(\bbR^{d \times n})$ with respect to 
${\cal E}^H_T \cdot \bbP$.

\setcounter{equation}{0}
\section{Markovian quadratic BSDEs}
\label{sec:FBSDE}

In this section we consider BSDEs of the form 
\begin{align} \label{markovian}
Y_t = h(W_T) + \int_t^T \crl{F(s,W_s,Y_s,Z_s)+Z_s G(s,W_s,Y_s,Z_s)}ds-\int_t^TZ_sdW_s 
\end{align}
for measurable functions $h:\bbR^n\to\bbR^d$, 
$F:[0,T]\times\bbR^n\times\bbR^d\times\bbR^{d\times n}\to\bbR^d$ and 
$G:[0,T]\times\bbR^n\times\bbR^d\times\bbR^{d\times n}\to\bbR^n$.

The following theorem gives conditions under which \eqref{markovian} has a solution if 
there is a solution to a related FBSDE.
    
\begin{theorem} \label{mainthm0}
Assume that there exists a constant $C\in\bbR_+$ and a nondecreasing function
$\rho:\bbR_+\to\bbR_+$ such that the following conditions hold:
\begin{itemize}
\item[{\rm (A1)}] $|h(x)| \le C$
\item [{\rm (A2)}] $y^TF(t,x,y,z)\leq
    C|y|\brak{1+|y|+|z|}$ for all $(t,x,y,z)\in[0,T]\times\bbR^n\times\bbR^d\times\bbR^{d\times n}$
\item[{\rm (A3)}] $\abs{G(t,x,y,z)}\leq
  \rho(|y|)\brak{1+|z|}$ for all
  $(t,y,z)\in[0,T]\times\bbR^d\times\bbR^{d\times n}$
\item[{\rm (A4)}] The FBSDE
\begin{align*}
dP_t&= G(t,P_t,Q_t,R_t)dt+dW_t, \quad P_0=0\\
dQ_t&=-F(t,P_t,Q_t,R_t)dt+ R_tdW_t, \quad Q_T=h(P_T)
\end{align*}
has a solution $(P,Q,R)\in\bbH^2(\bbR^n)\times\bbS^2(\bbR^d)\times\bbH^2(\bbR^{d\times n})$ 
such that $Q_t=q(t,P)$ and $R_t=r(t,P)$ for predictable functions
$q : [0,T] \times C([0,T],\bbR^n) \to\bbR^d$ and $r :[0,T] \times C([0,T],\bbR^n)\to\bbR^{d\times n}$.
\end{itemize}
Then $(Y_t,Z_t) =(q(t,W), r(t,W))$ is a solution of the BSDE
\eqref{markovian} in $\bbS^{\infty}(\mathbb{R}^d)
\times \bbH^{\rm BMO}(\bbR^{d \times n})$, and $Z$ is bounded if $R$ is bounded.
\end{theorem}

\begin{proof}
One obtains from It\^{o}'s formula that for every 
$a \in \bbR_+$ and $[0,T]$-valued stopping time $\tau$,
$$
e^{a\tau}|Q_\tau|^2= e^{aT}|h(P_T)|^2
+\int_\tau^Te^{as}\brak{2Q^T_sF(s,P_s,Q_s,R_s)-|R_s|^2-a|Q_s|^2}ds 
- \int_\tau^T 2e^{as}Q^T_s R_s dW_s.
$$
Since $Q \in \bbS^2(\bbR^d)$ and $R \in \bbH^2(\bbR^{d\times n})$, one has
$$
\bbE \sqrt{\int_0^T |Q^T_s R_s|^2 ds} \le
\bbE \sup_{0 \le s \le T} |Q_s| \sqrt{\int_0^T |R_s|^2 ds}
\le \norm{Q}_{\bbS^2} \norm{R}_{\bbH^2} < \infty.
$$
So it follows from the Burkholder--Davis--Gundy inequality that
$\sup_{0 \le t \le T} |\int_0^t 2e^{as} Q^T_s R_s dW_s|$ is integrable, implying that
the local martingale $\int_0^t 2e^{as} Q^T_s R_s dW_s$ is a true martingale. Therefore,
$$
e^{a\tau}|Q_\tau|^2=\bbE_\tau\brak{e^{aT}|h(P_T)|^2
+\int_\tau^Te^{as}\brak{2Q^T_sF(s,P_s,Q_s,R_s)-|R_s|^2-a|Q_s|^2}ds}.
$$
By assumption (A2),
\begin{align*}
& 2 Q^T_sF(s,P_s,Q_s,R_s)-|R_s|^2-a|Q_s|^2 \leq2C|Q_s|(1+|Q_s|+|R_s|)-|R_s|^2-a|Q_s|^2\\
&\leq C^2+(2C^2+2C+1-a)|Q_s|^2-\frac{1}{2}|R_s|^2.
\end{align*}
So for $a=2C^2+2C+1$, one obtains
\begin{align*}
& |Q_\tau|^2+\frac{1}{2} \bbE_\tau\int_\tau^T |R_s|^2 ds \le
e^{a\tau}|Q_\tau|^2+\frac{1}{2}\bbE_\tau\int_\tau^Te^{as}|R_s|^2ds\\ & \leq
\bbE_\tau\brak{e^{aT}|h(P_T)|^2+C^2\int_\tau^Te^{as}ds}\leq C^2e^{aT}(1+T).
\end{align*}
In particular, $Q$ is in $\bbS^{\infty}(\mathbb{R}^d)$ and $R$ in $\bbH^{\rm BMO}(\mathbb{R}^{d \times n})$.
By assumption (A3), one has
\begin{align*}
|G(s,P_s,Q_s,R_s)| \le \rho(C^2e^{aT}(1+T)) (1+|R_s|), 
\end{align*}
from which it follows that $G(s,P_s,Q_s,R_s)$ belongs to $\bbH^{\rm BMO}(\bbR^{n\times 1})$.
Therefore, $P$ is a Brownian motion under the measure
${\cal E}^{-G}_T \cdot \mathbb{P}$, and $R$ is still in $\bbH^{\rm BMO}(\mathbb{R}^{d \times n})$ under 
${\cal E}^{-G}_T \cdot \mathbb{P}$. The backward equation in (A4) can be written as
\begin{align*}
dQ_t = -\brak{F(t,P_t,Q_t,R_t)+ R_t G(t,P_t,Q_t,R_t)}dt+R_tdP_t,\quad Q_T=h(P_T).
\end{align*}
But since $Q_t=q(t,P)$ and $R_t=r(t,P)$, one has
\[
d q(t,P)=-\brak{F(t,P_t,q(t,P),r(t,P))+ r(t,P) G(t,P_t,q(t,P),r(t,P))}dt+r(t,P) dP_t.
\]
So $(Y,Z)=\brak{q(\cdot,W),r(\cdot,W)}$ is in $\bbS^{\infty}(\mathbb{R}^d)
\times \bbH^{\rm BMO}(\mathbb{R}^{d \times n})$ and satisfies
\[
dY_t= - \brak{F(t,W_t,Y_t,Z_t)+Z_t G(t,W_t,Y_t,Z_t)}dt+Z_tdW_t, \quad Y_T=h(W_T).
\]
Finally, if $R$ is bounded, then so is $Z$.
\end{proof}


\begin{Remark}
Since the BSDE \eqref{markovian} is Markovian, it is related to 
the semilinear parabolic PDE with terminal condition
\[
u_t+\frac{1}{2}\triangle u+F(t,x,u,\nabla u)+(\nabla u)g(t,x,u,\nabla u)=0,\quad u(T,x)=h(x).
\]
For example, if the PDE has a $C^{1,2}$-solution $u:[0,T]\times\bbR^n\to\bbR^d$,
it follows from It\^{o}'s formula that 
$(Y_t,Z_t)=(u(t,W_t),\nabla u(t,W_t))$ solves the BSDE \eqref{markovian}. But 
the standard construction of a viscosity solution to the PDE from a BSDE solution 
does not work because the necessary comparison results do not extend from the 
one- to the multidimensional case; see Peng (1999).
\end{Remark}

A crucial assumption of Theorem \ref{mainthm0} is {\rm (A4)}. There exist different results 
in the FBSDE literature from which it follows. The following proposition derives the 
existence of a unique solution to the quadratic BSDE \eqref{markovian} 
from an FBSDE result of Delarue (2002).

\begin{proposition} \label{prop:del}
Assume there exists a constant $C \in\bbR_+$ such that for all $t,x,x',y,y',z,z'$ the following hold:
\begin{align*}
& \bullet \; \abs{F(t,x,y,z)-F(t,x',y',z')} \leq C(|x-x'|+|y-y'|+|z-z'|)\\
& \bullet \; |G(t,x,y,z)-G(t,x',y',z')| \leq C(|x-x'|+|y-y'|+|z-z'|)\\
& \bullet \; |h(x)-h(x')| \leq C|x-x'|\\
& \bullet \; |F(t,x,0,0)| + |G(t,x,0,0) + |h(x)| \leq C.
\end{align*}
Then the BSDE \eqref{markovian} has a unique solution $(Y,Z)$ in 
$\bbS^{\infty}(\bbR^d) \times \bbH^{\infty}(\bbR^{d\times n})$, and it is of the form 
$Y_t = y(t,W_t)$, $Z_t = \nabla_x y (t,W_t)$, where 
$y : [0,T] \times \mathbb{R}^n \to \mathbb{R}^d$ is a continuous 
function that is uniformly Lipschitz in $x \in \mathbb{R}^n$ and $\nabla_x$ 
denotes the weak derivative with respect to $x$ in the Sobolev sense.
\end{proposition}

\begin{proof}
It follows from the assumptions that the conditions of Proposition 2.4
and Theorem 2.6 of Delarue (2002) hold.
Therefore, the FBSDE in {\rm (A4)} of Theorem \ref{mainthm0} has a unique solution $(P,Q,R)$
in $\bbH^2(\bbR^n) \times \bbS^{\infty}(\bbR^d) \times \bbH^{\infty}(\bbR^{d \times n})$ such that 
$Q$ is of the form $Q_t = q(t,P_t)$ for a bounded continuous function 
$q : [0,T] \times \mathbb{R}^n \to \mathbb{R}^d$ which is uniformly Lipschitz in $x \in \mathbb{R}^n$.
As a consequence, the process $G_t = G(t,P_t,Q_t,R_t)$ is bounded, 
from which it follows that ${\cal E}^{-G}_T \cdot \mathbb{P}$
defines a probability measure equivalent to $\bbP$, under which $P$ is an $\bbF$-adapted 
$n$-dimensional Brownian motion. It can be seen from the representation
$$
Q_t = q(t,P_t) = Q_0 - \int_0^t \crl{F(s,P_s, Q_s,R_s)+ R_s G(s,P_s,Q_s,R_s)} ds +\int_0^t R_s dP_s
$$
that $Q$ is a continuous $\bbF$-semimartingale. By Stricker's theorem, it is also a 
continuous semimartingale with respect to the filtration $\bbF^P$
generated by $P$. In particular, it has a unique canonical $(\bbF^P, \tilde{\bbP})$-semimartingale
decomposition $Q_t= Q_0+ M_t +A_t$, where $M$ is a
continuous $(\bbF^P,\tilde\bbP)$-local martingale and $A$ a finite variation
process with $M_0=A_0=0$. By the martingale representation theorem, $M_t$ can be written as
$M_t = \int_0^t H_s dP_s$ for a unique $\bbF^P$-predictable process $H$.
But since $P$ is an $\bbF$-Brownian motion under $\tilde{\bbP}$,
$Q_t= Q_0+ M_t +A_t$ is also the unique canonical $(\bbF, \tilde{\bbP})$-semimartingale decomposition of $Q$.
It follows that $R = H$, and therefore, $R_t = r(t,P)$ for a predictable function 
$r :[0,T]\times C([0,T],\bbR^n) \to \bbR^{d\times n}$. This shows that {\rm (A4)} holds. 
Since (A1)--(A3) of Theorem \ref{mainthm0} are satisfied as well, one obtains that
$(Y_t,Z_t) = (q(t,W_t),r(t,W))$ is a solution of the BSDE \eqref{markovian}.
Moreover, since $q$ is continuous and $q(t,P_t)$ an {\Ito} process, it follows from 
Theorem 1 of Chitashvili and Mania (1996) that $r(t,P)= \nabla_xq(t,P_t)$, where $\nabla_x q$ is a bounded
weak derivative of $q$ with respect to $x$ in the Sobolev sense. This shows that 
$(q(t,W_t), \nabla_x q(t,W_t))$ is a solution of  \eqref{markovian} 
in $\bbS^{\infty}(\bbR^d) \times \bbH^{\infty}(\bbR^{d \times n})$.

Now assume $(\tilde{Y}, \tilde{Z})$ is another solution of \eqref{markovian} 
in $\bbS^{\infty}(\bbR^d) \times \bbH^{\infty}(\bbR^{d \times n})$ and 
let $L$ be a common bound for $Y,Y',Z, Z'$. Then 
$(Y,Z)$ and $(Y',Z')$ are both solutions of the modified BSDE 
$$
Y_t= h(W_T) + \int_t^T f(s,W_s,\pi_L(Y_s,Z_s))ds - \int_t^T Z_sdW_s,
$$
where
\begin{align*}
f(t,x,y,z) :=F(t,x,y,z)+zG(t,x,y,z) \quad \mbox{and} \quad
\pi_L(y,z):=\brak{\min\crl{1,L/|y|}y, \min\crl{1,L/|z|}z}.
\end{align*}
Since this BSDE satisfies the conditions of Pardoux and Peng (1990), it has a unique solution
in $\bbS^{\infty}(\bbR^d) \times \bbH^{\infty}(\bbR^{d \times n})$,
and it follows that $(Y,Z)=(\tilde{Y},\tilde{Z})$.
\end{proof}

\begin{Remark} \label{rem:del}
The assumptions of Proposition \ref{prop:del} can be slightly relaxed such that
the conditions of Theorem 2.6 of Delarue (2002) are still met. Then the same arguments 
yield the existence of a bounded solution $(Y,Z)$ to the BSDE \eqref{markovian}.
Uniqueness can be shown by using Pardoux (1999) instead of Pardoux and Peng (1990).
Alternatively, the assumptions of Proposition \ref{prop:del} can 
be modified such that they imply some other FBSDE result, such as
e.g., the one of Pardoux and Tang (1999).
\end{Remark}

\setcounter{equation}{0}
\section{Projectable quadratic BSDEs}
\label{sec:proj}

\begin{definition} 
We call a multidimensional BSDE projectable if its driver can be written as
\be \label{projectabledriver}
f(s,y,z)=P(s,a^Ty,a^Tz)+yQ(s,a^Ty,a^Tz)+zR(s,a^Ty,a^Tz) 
\ee
for a constant vector $a \in \bbR^d$ and predictable functions
\begin{align*}
&P: [0,T] \times \Omega \times\bbR\times\bbR^{1\times n}\to\bbR^{d\times 1}\\
&Q: [0,T] \times \Omega \times\bbR\times\bbR^{1\times n}\to\bbR\\
&R: [0,T] \times \Omega \times\bbR\times\bbR^{1\times n}\to\bbR^{n\times 1}.
\end{align*}
\end{definition} 

A projectable BSDE becomes one-dimensional if projected on the line
generated by $a\in \bbR^{d\times 1}$:
\begin{align*}
a^T Y_t = & a^T\xi+\int_t^T a^T \crl{P(s,a^TY_s,a^TZ_s) + Y_s Q(s,a^TY_s,a^TZ_s) + Z_s R(s,a^TY_s,a^TZ_s)}ds\\
& -\int_t^T a^TZ_s dW_s.
\end{align*}
In the following theorem we consider a projectable BSDE under conditions ensuring 
that the projected BSDE has a solution. This makes it possible to derive the existence of a solution to the 
multidimensional BSDE. 

\begin{theorem} \label{oned}
Consider a bounded terminal condition $\xi \in L^{\infty}({\cal F}_T)^d$ and a driver 
$f$ of the form \eqref{projectabledriver} such that 
$$
\abs{P(s, u,v)} \le C(1+|u|), \quad \abs{Q(s, u,v)} \le C, \quad \abs{R(s,u,v)} \le C + \rho(|u|)|v| 
$$
for a constant $C \in \mathbb{R}_+$ and a nondecreasing function
$\rho:\bbR_+\to \bbR_+$. Then the BSDE 
\be \label{orBSDE}
Y_t = \xi + \int_t^T f(s,Y_s,Z_s) ds - \int_t^T Z_s dW_s
\ee
has a solution
$(Y,Z) \in \bbS^\infty(\bbR^d)\times \bbH^{\rm BMO}(\bbR^{d \times n})$. 

Moreover, if
\[
F(s,u,v)= a^T P(s,u,v) + u Q(s, u,v)+ v R(s, u,v)
\]
satisfies
\be \label{Lip}
|F(s,u,v) - F(s,u',v') \le C|u-u'| + C (1+ |v| \vee |v'|)|v-v'|,
\ee
then \eqref{orBSDE} has only one solution $(Y,Z)$ in $\bbS^\infty(\bbR^d)\times \bbH^{\rm BMO}(\bbR^{d \times n})$.
\end{theorem}

\begin{proof}
By Theorem 2.3 of Kobylanski (2000), the one-dimensional BSDE
\begin{align}\label{1dBSDE}
U_t = a^T \xi + \int_t^T F \brak{s, U_s, V_s}ds - \int_t^T V_s dW_s
\end{align}
has a solution $(U,V) \in \bbS^{\infty}(\bbR) \times
\bbH^2(\bbR^{1 \times n})$. Therefore,
there exists a constant $K \ge 0$ such that $|F(s,U_s,V_s)| \le K(1+ |V_s|^2)$,
and it follows like in the proof of Proposition 2.1 in Briand and Elie (2013) that 
$V$ is in $\bbH^{\rm BMO}(\mathbb{R}^{1 \times n})$. 
Denote $P_s := P(s,U_s,V_s)$, $Q_s := Q(s,U_s,V_s)$, $R_s = R(s,U_s,V_s)$ and
assume the multidimensional linear BSDE 
\be \label{givenBSDE}
Y_t = \xi + \int_t^T (P_s + Y_s Q_s + Z_s R_s) ds - \int_t^T Z_s dW_s
\ee
has a solution 
$(Y,Z) \in \bbS^{\infty}(\mathbb{R}^d) \times \bbH^{\rm BMO}(\mathbb{R}^{d \times n})$.
It follows from the assumptions that $|R_s| \le C + \rho(\|U\|_{\bbS^\infty})|V_s|$.
So $R$ is in $\bbH^{\rm BMO}(\mathbb{R}^{n\times 1})$, and $\tilde{\bbP} := {\cal E}^R_T \cdot \bbP$ is an equivalent 
probability measure under which $\tilde{W}_t = W_t - \int_0^t R_s ds$
is a Brownian motion.
Now one can write
\be \label{YW}
Y_t = \xi +  \int_t^T (P_s + Y_s Q_s) ds - \int_t^T Z_s d\tilde{W}_s,
\ee
from which it follows that
$$
e^{\int_0^t Q_u du} Y_t = e^{\int_0^T Q_u du} \xi + \int_t^T e^{\int_0^s Q_u du} P_s ds 
- \int_t^T e^{\int_0^s Q_u du} Z_s d\tilde{W}_s,
$$
and therefore,
\be \label{QE}
e^{\int_0^t Q_s ds} Y_t = \tilde{\bbE}_t 
\edg{e^{\int_0^T Q_u du} \xi + \int_t^T e^{\int_0^s Q_u du} P_s ds},
\ee
where $\tilde{\bbE}$ denotes expectation with respect to $\tilde{\bbP}$.
This uniquely determines $Y$. Now $Z$ is uniquely given by \eqref{YW}.
To show that \eqref{givenBSDE} has a solution in 
$\bbS^{\infty}(\mathbb{R}^d) \times \bbH^{\rm BMO}(\mathbb{R}^{d \times n})$, one can
define $Y$ by \eqref{QE}, which is equivalent to 
$$
\Gamma_t Y_t = \bbE_t \edg{\Gamma_T \xi + \int_t^T \Gamma_s P_s ds},
$$
where $\Gamma$ is the unique solution of the SDE 
$$
d\Gamma_t = \Gamma_t (Q_s ds + R^T_s dW_s), \quad \Gamma_0 = 1.
$$
Then $Y$ belongs to $\mathbb{S}^{\infty}(\mathbb{R}^d)$, and by the martingale representation theorem, 
there exists a unique predictable process $Z$ such that $\Gamma Z$ belongs to 
$\bbH^2(\mathbb{R}^{d \times n})$ and
$$
\int_0^T \Gamma_s (Y_s R^T_s + Z_s) dW_s = 
\Gamma_T \xi + \int_0^T \Gamma_s P_s ds - \bbE \edg{\Gamma_T \xi + \int_0^T \Gamma_s P_s ds}.
$$
Since $Y_0 = \bbE \edg{\Gamma_T \xi + \int_0^T \Gamma_s P_s ds}$, one has
\[
Y_0 + \int_0^t \Gamma_s (Y_sR^T_s +Z_s) dW_s = \bbE_t \edg{\Gamma_T \xi + \int_0^T \Gamma_s P_s ds}
= \Gamma_t Y_t + \int_0^t \Gamma_s P_sds.
\]
Therefore,
$$
Y_t = \Gamma^{-1}_t \brak{Y_0 + \int_0^t \Gamma_s (Y_sR^T_s +Z_s) dW_s -  \int_0^t \Gamma_s P_s ds},
$$
then one obtains
$$
dY_t = - (P_t + Y_t Q_t + Z_t R_t) dt + Z_t dW_t, \quad Y_T = \xi.
$$
In particular, \eqref{YW} holds, from which it can be seen that $M_t = \int_0^t Z_s d\tilde{W}_s$ is a bounded 
$\tilde{\bbP}$-martingale. Since $\tilde{\bbE}_{\tau} \int_{\tau}^T |Z|^2_s ds = \tilde{\bbE}_{\tau} (M_T-M_{\tau})^2$, 
this shows that $Z$ is in $\bbH^{\rm BMO}(\mathbb{R}^{d \times n})$ with respect to 
$\tilde{\bbP}$ and hence, also with respect to $\bbP$. So we have shown that for a given solution $(U,V) \in
\bbS^{\infty}(\bbR) \times \bbH^{\rm BMO}(\bbR^{1 \times n})$
of the one-dimensional BSDE \eqref{1dBSDE}, the linear BSDE \eqref{givenBSDE} has a unique 
solution $(Y,Z) \in \bbS^{\infty}(\mathbb{R}^d) \times \bbH^ {\rm BMO}(\mathbb{R}^{d \times n})$.
$(a^T Y, a^T Z)$ solves the one-dimensional linear BSDE 
$$
\tilde{U}_t = a^T \xi + \int_t^T
(a^T P_s + \tilde{U}_s Q_s + \tilde{V}_s R_s) ds - \int_t^T \tilde{V}_s dW_s,
$$
which, like \eqref{givenBSDE}, can be shown to have a unique solution in
$\bbS^{\infty}(\bbR) \times \bbH^{\rm BMO}(\bbR^{1\times n})$. It follows that
$(a^T Y, a^T Z) = (U,V)$, from which one obtains that $(Y,Z)$ solves
the original BSDE \eqref{orBSDE}.

If the additional condition \eqref{Lip} holds, one obtains from Theorem 2.6 in 
Kobylanski (2000) that the one-dimensional BSDE \eqref{1dBSDE} admits only one solution 
$(U,V)$ in $\bbS^{\infty}(\mathbb{R}) \times \bbH^{\rm BMO}(\mathbb{R}^{1\times n})$, and it follows that
\eqref{orBSDE} has a unique solution $(Y,Z)$ in $\bbS^{\infty}(\mathbb{R}^d) \times \bbH^{\rm BMO}(\mathbb{R}^{d \times n})$.
\end{proof}

\setcounter{equation}{0}
\section{Subquadratic BSDEs}
\label{sec:sub}

In the case where the BSDE \eqref{bsde} is not Markovian or projectable, we assume the driver 
$f(s,y,z)$ to be of strictly subquadratic growth in $z$. 
For constants $C, C_i \in \mathbb{R}_+$, $\varepsilon \in (0,1)$ and a nondecreasing function 
$\rho:\bbR_+\to\bbR_+$, consider the following conditions:
\begin{itemize}
\item[(B1)] $\xi$ is a $d$-dimensional $\scF_T$-measurable random vector satisfying $|\xi| \le C$
\item[(B2)] $|f(s,y,z)| \le C\brak{1+|y|+\rho(|y|)|z|^{2-\varepsilon}}$
\item[(B3)] 
$\abs{f(t,y,z)-f(t,y',z')} \le \rho\brak{|y|\vee|y'|}\brak{|y-y'|+ \brak{1+\brak{|z|\vee|z'|}^{1-\varepsilon}}|z-z'|}$
\item[(B4)] $f(s,y,z)=F(s,y,z)+G(s,y,z)$ and 
$$
y^T F(s,y,z) \le C |y|(1+|y|+|z|), \quad y^T G(s,y,z) \leq |y^T z| \rho(|y|)|z|.
$$
\end{itemize}
The main result of this section is the following

\begin{theorem} \label{thmsub}
A BSDE 
\be \label{bsdesub}
Y_t = \xi + \int_t^T f(s,Y_s,Z_s) ds - \int_t^T Z_s dW_s
\ee
satisfying {\rm (B1)--(B4)} has a unique solution
$(Y,Z)$ in $\bbS^\infty(\bbR^d)\times \bbH^{\rm BMO}(\bbR^{d \times n})$, and
$$
|Y_t| \le (C+1) \exp \brak{\frac{(C+1)^2}{2} (T-t)}.
$$
\end{theorem}

We prove Theorem \ref{thmsub} by first showing that the BSDE \eqref{bsdesub}
has a unique solution for short time intervals and then constructing a 
solution on $[0,T]$ recursively backwards in time.

For small $h >0$ we use Banach's fixed point theorem to prove the existence 
and uniqueness of a solution on $[T-h, T]$. For $R \in \mathbb{R}_+$, define
$$
\scB_R:=\crl{(Y,Z)\in \bbS^{\infty}_{[T-h, T]}(\bbR^d)\times
\bbH^{\rm BMO}_{[T-h, T]}(\bbR^{d\times n}): \norm{Y}_{\bbS^\infty_{[T-h, T]}} \le R,
\norm{Z}_{{\rm BMO}_{[T-h, T]}} \le R}.
$$

\begin{lemma} \label{lemma:local}
Assume that $R \ge 3C$ and {\rm (B1)--(B3)} hold. Then there exists a constant $\delta>0$ only 
depending on $C$, $R$ and $\rho(R)$ such that for all $h \in (0, \delta]$, the BSDE \eqref{bsdesub} 
has a unique solution $(Y,Z) \in \scB_R$ on $[T-h, T]$.
\end{lemma}

\begin{proof}
In the whole proof we assume $t \in [T-h, T]$ and treat 
$\rho=\rho(R)$ as a constant. This is possible because $Y$ will turn out to be bounded by $R$.
For $(y,z)\in \scB_R$, define $\phi(y,z) := (Y,Z)$, where 
$$
Y_t = \bbE_t \brak{\xi + \int_t^T f(s,y_s,z_s)ds}, \quad T -h \le t \le T,
$$
and $Z$ is the unique $\mathbb{R}^{d \times n}$-valued predictable process satisfying
$$
\int_{T-h}^T Z_s dW_s = \xi + \int_{T-h}^Tf(s,y_s,z_s)ds - \bbE_{T-h} \brak{\xi + \int_{T-h} ^T f(s,y_s,z_s)ds}.
$$
Then
$$
|Y_t| \le C + \bbE_t \int_t^T|f(s,y_s,z_s)|ds, \quad T - h \le t \le T,
$$
and if $\tau$ is a stopping time with values in $[T-h,T]$, it follows from (B2) and H\"older's inequality that
\begin{eqnarray*}
&& \bbE_{\tau}\int_{\tau}^T|f(s,y_s,z_s)|ds \le C\bbE_{\tau} \int_{\tau}^T \brak{1+|y_s|+\rho|z_s|^{2 -\varepsilon}}ds\\
&\le& C h \brak{1+\norm{y}_{\bbS^\infty_{[T-h, T]}}} + C \rho
\bbE_{\tau} \int_{\tau}^T|z_s|^{2-\varepsilon}ds \le C h \brak{1+\norm{y}_{\bbS^\infty_{[T-h, T]}}}
+ C \rho h^{\varepsilon/2}\norm{z}_{\rm BMO_{[T-h, T]}}^{2-\veps}\\
&\le& Ch\brak{1+R} + C \rho h^{\varepsilon/2}R^{2-\veps}.
\end{eqnarray*}
Choose $h> 0$ so small that $Ch\brak{1+R} + C \rho h^{\varepsilon/2}R^{2-\veps} \le C$. 
Then 
$$
\bbE_{\tau}\int_{\tau}^T|f(s,y_s,z_s)|ds \le C \quad \mbox{and} \quad
\norm{Y}_{\bbS_{[T-h, T]}^\infty} \le 2C \le R.
$$
On the other hand, it follows from the definition of $(Y,Z)$ that 
\begin{equation}\label{ylipbsde}
Y_t = \xi+\int_t^Tf(s,y_s,z_s)ds - \int_t^TZ_sdW_s, \quad T - h \le t \le T.
\end{equation}
So one obtains from It\^{o}'s formula,
\begin{align*}
 |Y_{\tau}|^2+\int_{\tau}^T
 |Z_s|^2ds=|\xi|^2+2\int_{\tau}^T Y_s^Tf(s,y_s,z_s)ds-2\int_{\tau}^TY_s^TZ_sdW_s.
\end{align*}
Since $(y,z) \in \bbS^{\infty}_{[T-h, T]}(\bbR^d)\times \bbH^{\rm BMO}_{[T-h, T]}(\bbR^{d\times n})$,
it follows from (B1) and (B2) that $\xi + \int_{T - h}^T f(s,y_s,z_s)ds$ is $p$-integrable for some $p > 1$.
By Doob's maximal $L^p$-inequality, $\sup_{T-h \le t \le T} |\int_{T-h}^t Z_s dW_s|$ 
is also $p$-integrable. So one obtains from the Burkholder--Davis--Gundy inequality that
$\int_{T-h}^T |Y^T_s Z_s|^2 ds$ is $p/2$-integrable
and $\sup_{T-h \le t \le T} |\int_{T-h}^t Y^T_s Z_s dW_s|$ is $p$-integrable.
This implies that the local martingale $\int_{T-h}^t Y^T_s Z_s dW_s$ is a true martingale. 
Therefore,
\begin{align*}
\bbE_{\tau} \int_{\tau}^T |Z_s|^2ds \le
C^2+ 2 \norm{Y}_{\bbS_{[T-h, T]}^\infty} \bbE_{\tau} \int_{\tau}^T|f(s,y_s,z_s)|ds \le  5 C^2,
\end{align*}
which shows that $\norm{Z}_{{\rm BMO}_{[T-h, T]}} \le 3 C \le R$. 
So $\phi$ maps $\scB_R$ into itself.

Next, we show that $\phi$ is a contraction on $\scB_R$. Choose $(y,z),(y',z') \in \scB_R$ and 
denote $(Y,Z) = \phi (y,z)$, $(Y',Z') = \phi (y',z')$, $\Delta y=y-y'$, $\Delta z=z-z'$,
$\Delta Y=Y-Y'$ and $\Delta Z=Z-Z'$. Then,
\be \label{DeltaY}
\Delta Y_t=\int_t^T\brak{f(s,y_s,z_s)-f(s,y'_s,z'_s)}ds-\int_t^T\Delta Z_sdW_s,
\ee
and by {\Ito}'s formula,
\be \label{DeltaZ}
|\Delta Y_{\tau}|^2+\int_{\tau}^T|\Delta Z_s|^2ds =2\int_{\tau}^T\Delta Y_s^T
\brak{f(s,y_s,z_s)-f(s,y'_s,z'_s)}ds-2\int_{\tau}^T\Delta Y_s^T\Delta Z_s dW_s.
\ee
It follows from (B3) that
\begin{align*}
& \bbE_{\tau} \int_{\tau}^T\abs{f(s,y_s,z_s)-f(s,y'_s,z'_s)}ds \le \bbE_{\tau}
\int_{\tau}^T\rho \brak{|\Delta y_s|+\brak{1+\brak{|z_s|\vee|z'_s|}^{1-\veps}}|\Delta z_s|}ds\\
&\le \rho h \norm{\Delta y}_{\bbS_{[T-h, T]}^\infty}+\rho\sqrt{\bbE_\tau\int_\tau^T\brak{1+\brak{|z_s|\vee|z'_s|}^{1-\veps}}^2ds}\sqrt{\bbE_\tau\int_\tau^T
\abs{\Delta z_s}^2ds}\\
&\leq \rho h\norm{\Delta y}_{\bbS_{[T-h, T]}^\infty}
+\rho\sqrt{2\brak{h+\bbE_\tau\int_\tau^T\brak{|z_s| + |z'_s|}^{2-2\veps}ds}} \norm{\Delta
z}_{{\rm BMO}_{[T-h,T]}}\\
&\le \rho h\norm{\Delta y}_{\bbS_{[T-h, T]}^\infty}
+\rho \sqrt{2 \brak{h + h^\veps (2R)^{2-2\veps}}}\norm{\Delta z}_{{\rm BMO}_{[T-h,T]}}.
\end{align*}
So for $h >0$ small enough, one has
\begin{align*}
& \bbE_{\tau} \int_{\tau}^T\abs{f(s,y_s,z_s)-f(s,y'_s,z'_s)}ds
\le \frac{1}{4} \brak{\norm{\Delta y}_{\bbS^\infty_{[T-h,T]}} + \norm{\Delta z}_{{\rm BMO}_{[T-h,T]}}}\\
& \le \frac{1}{2} \brak{\norm{\Delta y}_{\bbS^\infty_{[T-h,T]}} \vee \norm{\Delta z}_{{\rm BMO}_{[T-h,T]}}}.
\end{align*}
Since $\int_{T-h}^t \Delta Z_s dW_s$ and $\int_{T-h}^t \Delta Y^T_s \Delta Z_s dW_s$ are 
martingales, one obtains from \eqref{DeltaY},
$$
\norm{\Delta Y}_{\bbS_{[T-h, T]}^\infty} \le \frac{1}{2} 
\brak{\norm{\Delta y}_{\bbS_{[T-h, T]}^\infty} \vee \norm{\Delta z}_{{\rm BMO}_{[T-h,T]}}},
$$
and from \eqref{DeltaZ},
\begin{align*}
& \norm{\Delta Z}^2_{{\rm BMO}_{{[T-h,T]}}} \le 2 \norm{\Delta Y}_{\bbS^\infty_{[T-h,T]}} 
\sup_{\tau} \bbE_{\tau} \int_{\tau}^T\abs{f(s,y_s,z_s)-f(s,y'_s,z'_s)}ds\\
& \le \frac{1}{2} \brak{\norm{\Delta y}_{\bbS^\infty_{[T-h,T]}} \vee \norm{\Delta z}_{{\rm BMO}_{[T-h,T]}}}^2.
\end{align*}
This shows that there exists a $\delta > 0$ only depending on $C$, $R$ and $\rho(R)$ such that for 
$h \in (0,\delta]$, $\phi$ is a contraction on ${\cal B}_R$, from which it follows that on $[T-h, T]$, 
\eqref{bsdesub} has a unique solution in ${\cal B}_R$.
\end{proof}

In the next step we show that under the additional condition (B4), the solution of Lemma
\ref{lemma:local} satisfies 
$$|Y_t| \le (C+1) \exp \brak{\frac{(C+1)^2}{2} (T-t)}.
$$
This allows us to construct a solution on the whole time interval $[0,T]$ recursively.

\begin{lemma} \label{lemma:bound}
Assume {\rm (B1)} and {\rm (B4)} hold, and for some $h >0$, the BSDE \eqref{bsdesub} has a solution
$(Y,Z) \in \bbS^{\infty}_{[T- h,T]}(\mathbb{R}^d) 
\times \bbH^{\rm BMO}_{[T-h,T]}(\mathbb{R}^{d \times n})$. Then
$$
|Y_t| \le (C+1) \exp \brak{\frac{(C+1)^2}{2} (T-t)}
$$
for $t\in[T-h,T]$.
\end{lemma}

\begin{proof}
Fix $a\in\bbR_+$ and note that for all $i = 1,\dots, d$, one obtains from It\^{o}'s formula,
\begin{align*}
& e^{at} |Y_t|^2 = e^{aT} |\xi|^2 + 2 \int_t^Te^{as} \brak{Y_s^T f(s,Y_s,Z_s) - \frac{a}{2}
|Y_s|^2 - \frac{1}{2} |Z_s|^2}ds - 2\int_t^Te^{as}Y^T_sZ_sdW_s\\
&= e^{aT} |\xi|^2 + 2 \int_t^Te^{as} \brak{Y^T_sf(s,Y_s,Z_s)
- \frac{a}{2} |Y_s|^2 - \frac{1}{2} |Z_s|^2 - Y^T_s Z_s \brak{\rho(|Y_s|)|Z_s| \frac{Y^T_s Z_s}{|Y^T_s Z_s|}}^T}ds\\
& \quad - 2 \int_t^Te^{as}Y_s^T Z_sd\tilde W_s,
\end{align*}
where 
\[
\tilde W_s=W_s-\int_0^t\brak{\rho(|Y_s|)|Z_s|\frac{Y_s^T Z_s}{|Y_s^T Z_s|}}^Tds
\]
is a Brownian motion under the equivalent probability measure 
$$
\tilde{\bbP} := \exp \brak{\int_0^T \rho(|Y_s|)|Z_s|\frac{Y_s^T Z_s}{|Y_s^T Z_s|} dW_s
- \frac{1}{2} \int_0^T \brak{\rho(|Y_s|)|Z_s|}^2 ds} \cdot \bbP.
$$
Denote by $\tilde\bbE$ the expectation with respect to $\tilde\bbP$ and notice that
\begin{align*}
&Y_s^T f(s,Y_s,Z_s) - \frac{a}{2}
|Y_s|^2 - \frac{1}{2}|Z_s|^2 -Y_s^TZ_s\brak{\rho(|Y_s|)|Z_s|\frac{Y_s^T Z_s}{|Y_s^T Z_s|}}^T\\
&= Y_s^T F(s,Y_s,Z_s) + Y_s^T G(s,Y_s,Z_s) - \frac{a}{2}
|Y_s|^2 - \frac{1}{2}|Z_s|^2 - |Y_s^TZ_s| \rho(|Y_s|)|Z_s|\\
& \leq C |Y_s| (1 +|Y_s|+ |Z_s|) - \frac{a}{2} |Y_s|^2-\frac{1}{2}|Z_s|^2\\
&\leq \frac{C^2}{2} + \frac{1}{2} \brak{(C+1)^2 - a}|Y_s|^2.
\end{align*}
So for $a= (C+1)^2$, one obtains
\begin{align*}
e^{at}|Y_t|^2
&\leq
C^2e^{aT}+C^2 \int_t^Te^{as}ds - 2\int_t^Te^{as}Y_s^TZ_sd\tilde W_s,
\end{align*}
which by applying $\tilde{\bbE}_t$, gives
$|Y_t|^2 \le C^2 e^{a(T-t)} (a+1)/a$, and therefore,
\[
|Y_t| \le C \exp \brak{\frac{(C+1)^2}{2} (T-t)} \sqrt{\frac{C^2+1}{C^2}}
\le (C+1) \exp \brak{\frac{(C+1)^2}{2} (T-t)}.
\]
\end{proof}

\medskip 
\noindent
{\bf Proof of Theorem \ref{thmsub}}\\
Choose $h > 0$ such that $T/h \in\bbN$ and the mapping $\phi$ from the proof of Lemma
\ref{lemma:local} is a contraction when $C$ is replaced by $(C+1) \exp \brak{(C+1)^2T/2}$.
Then \eqref{bsdesub} has a unique solution $(Y^{(1)},Z^{(1)}) \in \bbS^{\infty}_{[T- h,T]}(\mathbb{R}^d) \times
\bbH^{\rm BMO}_{[T- h,T]}(\mathbb{R}^{d \times n})$ on $[T-h ,T]$. By
Lemma \ref{lemma:bound}, one has $|Y^{(1)}_t| \le \varphi(t) := (C+1) \exp \brak{(C+1)^2 (T-t)/2}$.
Another application of Lemma \ref{lemma:local} yields that on $[T-2h,T- h]$,
the BSDE \eqref{bsdesub} with terminal condition $Y^{(1)}_{T- h}$ has 
a unique solution $(Y^{(2)},Z^{(2)}) \in \bbS^{\infty}_{[T- 2h,T-h]}(\mathbb{R}^d) \times
\bbH^{\rm BMO}_{[T- 2 h,T-h]}(\mathbb{R}^{d \times n})$.
Pasting together the two solutions gives a solution 
$(Y,Z) \in \bbS^{\infty}_{[T- 2 h,T]}(\mathbb{R}^d) \times
\bbH^{\rm BMO}_{[T- 2 h,T]}(\mathbb{R}^{d \times n})$ on $[T-2 h,T]$. By Lemma
\ref{lemma:bound}, it satisfies $|Y_t| \le \varphi(t)$. Continuing like this yields a unique solution 
$(Y,Z) \in \bbS^{\infty}(\mathbb{R}^d) \times
\bbH^{\rm BMO}(\mathbb{R}^{d \times n})$ on $[0,T]$, which by 
Lemma \ref{lemma:bound}, satisfies $|Y_t| \le \varphi(t)$.
\qed
\section*{Acknowledgements}
We thank Peng Luo and two anonymous referees for helpful comments.


\begin{thebibliography}{25}
\bibitem{Ankirchner:2009p19192}{Ankirchner, S., Imkeller, P., Popier,
A. (2009). On measure solutions of backward stochastic
differential equations. Stoch. Proc. Appl. 119(9), 2744--2772.}
\bibitem{Bahlali:2001p24500}{Bahlali, K. (2001). Backward stochastic
differential equations with locally Lipschitz coefficient. C. R. Acad. Sci. Paris S\'er. I Math.
333(5), 481--486.}
\bibitem{Bahlali:2010us}{Bahlali, K., Essaky, E.H., Hassani, M.
(2010). p-integrable solutions to multidimensional BSDEs and
degenerate systems of PDEs with logarithmic
nonlinearities. Preprint arXiv:1007.2388, http://arxiv.org/abs/1007.2388.}
\bibitem{Briand:2013cu}{Briand, P., Elie, R. (2013). A simple
constructive approach to quadratic BSDEs with or without delay. Stoch. Proc. Appl. 123(8), 2921--2939.}
\bibitem{Cheridito:2012wm}{Cheridito, P., Nam, K. (2014). BSDEs with
terminal conditions that have bounded Malliavin derivative. J. Funct. Analysis, 266(3), 1257--1285.}
\bibitem{Mania:1996tj}{Chitashvili, R., Mania, M. (1996). Generalized
It\^o formula and derivation of Bellman's equation. Stochastic processes and related topics (Siegmundsberg, 1994), 1--21, 
Stochastics Monograph, 10, Gordon and Breach, Yverdon.}
\bibitem{Delarue:2002p25574}{Delarue, F. (2002). On the existence and
uniqueness of solutions to FBSDEs in a non-degenerate
case. Stoch. Proc. Appl. 99(2), 209--286.}
\bibitem{EPQ97}{El-Karoui, N., Peng, S., Quenez, M.C. (1997). Backward
stochastic differential equations in finance. Math. Finance, 7(1), 1--71.}
\bibitem{Anonymous:2012p19201}{Frei, C., Reis, Dos, G. (2011). A
financial market with interacting investors: does an equilibrium
exist? Math. Fin. Economics, 4(3), 161--182.}
\bibitem{HP06}{Hu, Y., Peng, S. (2006). On the comparison theorem for
multidimensional BSDEs. C. R. Math. Acad. Sci. Paris 343(2), 135--140.}
\bibitem{Kazamaki:1994ux}{Kazamaki, N. (1994). Continuous exponential
martingales and BMO, Lecture Notes in Mathematics, 1579. Springer-Verlag, Berlin.}
\bibitem{Kobylanski00}{Kobylanski, M. (2000). Backward stochastic
differential equations and partial differential equations with
quadratic growth. Ann. Prob. 28(2), 558--602.}
\bibitem{Liang:2010p19101}{Liang, G., Lionnet, A.,  Qian, Z. (2010). On
Girsanov's transform for backward stochastic differential
equations. Preprint arXiv:1011.3228, http://arxiv.org/abs/1011.3228.}
\bibitem{Mania:2005p25578}{Mania, M., Schweizer, M. (2005). Dynamic exponential utility indifference valuation. 
Ann. Appl. Prob. 15(3), 2113--2143.}
\bibitem{P99}{Pardoux, E. (1999). BSDEs, weak convergence and homogenization of
semilinear PDEs. Nonlinear analysis, differential equations and
control (Montreal, QC, 1998), 503--549.}
\bibitem{Pardoux:1990p20509}{Pardoux, E.,Peng, S. (1990). Adapted
solution of a backward stochastic differential equation, Systems Contr. Lett. 14(1), 55--61. }
\bibitem{Pardoux:1999uv}{Pardoux, E., Tang, S. (1999). Forward-backward
stochastic differential equations and quasilinear parabolic
PDEs. Prob. Th. Rel. Fields 114(2), 123--150.}
\bibitem{Peng:1999ws}{Peng, S. (1999). Open problems on backward stochastic differential equations. Control of distributed parameter and stochastic systems (Hangzhou, 1998), 265--273, Kluwer Acad. Publ., Boston, MA.}
\bibitem{Tevzadze:2008p17046}{Tevzadze, R. (2008). Solvability of
backward stochastic differential equations with quadratic
growth. Stoch. Proc. Appl. 118(3), 503--515.}
\end{thebibliography}
\end{document}